\documentclass[11pt,leqno]{article}
\usepackage{amsmath, amsfonts, theorem, 
}
\usepackage{latexsym}
\usepackage{graphicx}
\makeatletter
\def\@maketitle{\newpage
    \null
    \vskip .8truein
    \begin{center}%
     {\bf \@title \par}%
     \vskip 1.5em
     {\small
      \lineskip .5em
      \begin{tabular}[t]{c}\@author
      \end{tabular}\par}%
    \end{center}%
    \par
    \vskip .4truein}
\@addtoreset{equation}{section} \@addtoreset{theorem}{section}
\@addtoreset{lemma}{section} \@addtoreset{proposition}{section}
\@addtoreset{definition}{section} \@addtoreset{corollary}{section}
\@addtoreset{remark}{section}

\let\nn=\nonumber

\newcommand{\re}{{\mathbb R}}
\newcommand{\nat}{{\mathbb N}}

\let\ds=\displaystyle

\def\R{{\bf R}}
\def\N{{\bf N}}
\def\N{{\bf N}}

\newtheorem{theorem}{Theorem}[section]
\newtheorem{lemma}{Lemma}[section]
\newtheorem{proposition}{Proposition}[section]
\newtheorem{definition}{Definition}[section]
\newtheorem{corollary}{Corollary}[section]
\newtheorem{remark}{Remark}[section]
\newtheorem{example}{Example}[section]
  {\hfill$\Box$\bigskip\par}

\def\proof{\list{}{\setlength{\leftmargin}{0pt}
                      \parskip=0pt\parsep=0pt\listparindent=2em
                      \itemindent=0pt}\item[]\futurelet\testchar\@maybe}

\def\@maybe{\ifx[\testchar \let\next\@Opt
          \else \let\next\@NoOpt \fi \next}
\def\@Opt[#1]{{\it Proof of #1.\ }}\def\@NoOpt{{\it Proof.\ }}
\begin{document}
\title{\Large \bf The ergodic problem for some subelliptic operators with unbounded coefficients}

\author{{\large \sc Paola Mannucci, Claudio Marchi, Nicoletta Tchou}\\
 \rm Universit\`a degli Studi di Padova, Universit\'e de Rennes 1}

\maketitle


\begin{abstract}
\noindent 
We study existence and uniqueness of the invariant measure for a stochastic process
with degenerate diffusion,
whose infinitesimal generator is a linear 
 subelliptic operator in the whole space $\re^N$ with coefficients that may be unbounded.
Such a measure together with a Liouville-type theorem will play a crucial role in two applications: the ergodic problem studied through stationary problems with vanishing discount and the long time behavior of the solution to a parabolic Cauchy problem. In both cases, the constants will be characterized in terms of the invariant measure.
\end{abstract}
\noindent {\bf Keywords}:  Subelliptic equations, Heisenberg group, invariant measure, 
viscosity solutions, degenerate elliptic equations, ergodic problem, long time behavior.

\noindent  {\bf 2010 AMS Subject classification:} 35J70, 35H10, 35Q93, 49L25, 58F11.


\section{Introduction}

 
 This paper is devoted to study with pde's methods, the existence and uniqueness of the invariant measure of stochastic processes
with degenerate diffusion,
whose infinitesimal generators are linear
 subelliptic operators in the whole space $\re^N$ with coefficients that may be unbounded.
The invariant measures play a crucial role in ergodicity, homogenization and large time behaviour of the value function associated to the process. 
These methods, based on optimal control theory and pde's arguments, were introduced in the 80's by Bensoussan and developed until nowadays (see the monograph~\cite{B1} by Bensoussan and references therein).

We shall first tackle the case of the Heisenberg group as model problem; after we shall extend our techniques to other subelliptic operators.
In the Heisenberg case, we consider the stochastic dynamics
\begin{equation}\label{process}
dX_t=b(X_t)dt +\sqrt2\sigma(X_t)dW_t\quad \textrm{for }t\in(0,+\infty),\qquad X_0= x^0\in \re^3
\end{equation}
where, if $x=(x_1, x_2, x_3)\in \re^3$, the matrix $\sigma(x)$ has the form
\begin{equation}\label{matrixH}
\sigma(x)=\begin{bmatrix}
	1 &0\\
0& 1\\
2x_2& -2x_1
\end{bmatrix}
	\end{equation}
(in other words, the columns of $\sigma$ are vectors generating the Heisenberg group) while $W_t$ is a $3$-dimensional Brownian motion.

Our principal aim is to prove, under suitable assumptions on the drift $b$,  
the existence and uniqueness of the invariant measure~$m$ associated to the process~\eqref{process}.\\
Let us recall from \cite{B1} that a probability measure~$m$ on $\re^3$ is an {\it invariant measure} for process~\eqref{process} if, for each $u_0\in \mathbb L^\infty(\re^3)$, it satisfies
\begin{equation}\label{4.12}
\int_{\re^3} u(x,t)m(x)\, dx = \int_{\re^3} u_0(x)m(x)\, dx
\end{equation}
where $u(x,t)=\mathbb E_x(u_0(X_t))$ is the solution to the parabolic Cauchy problem
\begin{equation*}
\left\{
\begin{array}{ll}
\partial_t u+{\cal L}u=0&\qquad \textrm{in }(0,+\infty)\times \re^3\\
u(0,x)=u_0(x)&\qquad \textrm{on }\re^3
\end{array}\right.
\end{equation*}
and
\begin{equation}\label{op_l}
-{\cal L}u:=tr(\sigma(x)\sigma^T(x)D^2u(x))+b(x)\cdot Du(x)
\end{equation}
is the {\it infinitesimal generator} of process~\eqref{process}.

It is well known (see~\cite[Sect. II.4 and II.5]{B1}) that the density of the probability~$m$ (which, with a slight abuse of notation, we still denote by~$m$) solves
\begin{equation*}
{\cal L}^*m=0,\qquad \int_{\re^3} m\, dx=1\qquad\textrm{and}\qquad m\geq 0,
\end{equation*}
where ${\cal L}^*m$ is the adjoint operator
\begin{equation*}
{\cal L}^*m= -\sum_{i,j}\partial_{ij}({(\sigma\sigma^T)}_{ij}m) + \sum_{i}\partial_{i}(b_im).
\end{equation*}
\noindent In the framework of locally strongly elliptic operators, Has'min\-ski\v \i~\cite[Sect. IV.4]{Ha} (see also~\cite[Sect.8.2]{LB}) established the existence of an invariant measure provided that there exists a bounded set~$U$ with smooth boundary such that
\begin{equation}\label{h2}
\left\{\begin{array}{l}
\textrm{for any $x^0\in\re^N\setminus U$, the mean time $\tau$ at which the path \eqref{process} issuing}\\
\textrm{from $x^0$ reaches $U$ is finite and $\mathbb E_x \tau$ is locally finite.}
\end{array}\right.
\end{equation}

In our case this result does not apply because the matrix $A:=\sigma\sigma^T$ with $\sigma$ given by (\ref{matrixH})
is 
\begin{equation}\label{matrixA}
A(x)=\begin{bmatrix}
	1 &0&2x_2\\
0& 1&-2x_1\\
2x_2& -2x_1&4(x_1^2+x_2^2)
\end{bmatrix}
	\end{equation}
	and it is only positive semidefinite.\\
It is worth noticing (see~\cite{BCM, C, LM}) that a sufficient condition for property~\eqref{h2} is the existence of a {\it Lyapunov}-like function~$w$ which satisfies, for some positive constants~$k$ and~$R_0$
\begin{equation}\label{Lipintr}
{\cal L}w\geq k \quad \textrm{for }|x|\geq R_0\qquad\textrm{and}\qquad
w(x)\to +\infty\quad \textrm{as } |x|\to+\infty.
\end{equation}
As one can easily check the presence of the first order term is somehow 'crucial' for the existence of such a function. 
We will prove the existence of such Lyapunov function under suitable assumptions on the drift $b$ that include also the Ornstein-Uhlenbeck case (see \cite{LB} and Remark \ref{OU} below) where the operator is of the following type
$$
-{\cal L}u:=\,tr(\sigma(x)\sigma^T(x)D^2u(x))-\alpha x\cdot Du(x),\qquad \alpha>0.
$$

For ergodicity results based on probabilistic methods we refer to \cite{KP} and \cite{Kus} and the references therein.
The existence of a Lyapunov function is reminiscent of similar conditions (for instance, see: \cite[Sect. 8.2]{LB} and \cite{PV1,PV2,PV3}) called ``recurrence condition'' in the probabilistic jargon.

Ichihara and Kunita~\cite{IK} (see also~\cite{Ku}) proved the existence of an invariant measure for hypoelliptic processes as ~\eqref{process} which are constrained in a compact set. It is worth to recall that, in unbounded set the existence of an invariant measure may fail as it can be easily seen for~\eqref{process} with $b=0$ and $\sigma=I$.

In this paper we want to establish existence and uniqueness of an invariant measure for process~\eqref{process}, namely for a process with the following features: it lies in an unbounded set 
and its infinitesimal generator is simultaneously degenerate and with
unbounded coefficients.
To this end we shall use only pure analytical arguments.

It is important to stress that, in the Heisenberg case,
the principal part of $\mathcal Lu$ can be written as 
$\sum_{i=1}^{2}X_i^2u$ where 
$X_1$,  $X_2$, are the vector fields given by the columns of $\sigma$ and that 
they 
satisfy H\"ormander condition: $X_1$,  $X_2$, and their commutators of any order span $\re^3$
at each point $(x_1, x_2, x_3)\in\re^3$. In this case we have that $[X_1, X_2]=-4\partial_{x_3}$.
This
property will play a crucial role in this paper
since, as for the uniformly elliptic case, we have regularity, comparison and maximum principle (\cite{Bo}). 

The methods used in this work are strongly inspired by the lectures "Equations paraboliques et ergodicit\'e"
of P.L Lions at Coll\`ege de France (2014-15)~\cite{Lnt} and  by a unpublished manuscript by P.L. Lions and M. Musiela~\cite{LM} (see also the paper of Cirant~\cite{C} for similar arguments).\\
Actually, we shall consider the process
\begin{equation}\label{p_rho}
dX_t^\rho=b(X_t^\rho)dt +\sqrt2\sigma_\rho(X_t^\rho)dW_t,
\end{equation}
where $\sigma_\rho$ is the approximating  matrix of $\sigma$ in (\ref{matrixH}):
$$
\sigma_{\rho}(x)=\begin{bmatrix}
	1 &0&0\\
0& 1&0\\
2x_2& -2x_1&\rho
\end{bmatrix}
	$$
such that $A_{\rho}= \sigma_{\rho}\sigma^T_{\rho}$ is locally strictly positive, 
constrained in a bounded set $O_n$ suitably chosen.

Let us stress that, in our argument, it is not enough to approximate the matrix~$A$ with any non-degenerate matrix~$A_\rho$ but we also need that $A_\rho$ can be written as $\sigma_{\rho}\sigma^T_{\rho}$, where $\sigma_{\rho}$ is the diffusion matrix of a new underlying optimal control problem. 
This issue motivates the fact that in~\eqref{p_rho} a new Brownian motion appears.

Let us recall from \cite{B1} that the invariant measure~$m^n_\rho$ of this process solves
\begin{equation*}
{\cal L}^*_\rho m^n_\rho=0 \qquad \textrm{in } O_n 
\end{equation*}
coupled with a boundary condition of Neumann type, where 
\begin{equation*}
-{\cal L}_\rho (u)=\,tr(\sigma_\rho(x)\sigma^T_\rho(x)D^2u(x))+b(x)\cdot Du(x)
\end{equation*}
is an uniformly elliptic operator in  $O_n$. Letting $n\to+\infty$, we obtain and invariant measure $m_\rho$ for the process~\eqref{p_rho} in the whole space; letting $\rho\to0^+$, we get the desired invariant measure for~\eqref{process}. The Lyapunov function will play a crucial role in these limits: it will be used in order to prove that all the $m_\rho$'s and $m$ are really measures (in other words, that the $m_\rho^n$ and the $m_\rho$ do not ``disperse at infinity'').

Moreover in this paper we also establish a Liouville type result.
Similar result for semilinear operator without the drift term can be founded in the papers~\cite{BCDC1, BCDC2, CDC} and references therein; in all these papers the nonlinear zeroth order term is the key ingredient whereas, in our setting, the crucial contribution is due to the drift.

\vskip 2mm
We shall use the invariant measure and the Liouville property in two classic applications: an ergodic problem and 
the long time behaviour of a Cauchy problem. 
For the former problem we consider the family of equations
\begin{equation}\label{ergointro}
\delta u_{\delta}-tr(\sigma(x)\sigma^T(x)D^2 u_{\delta})-b(x)Du_{\delta}=F(x)\qquad \textrm{in }\re^3,
\end{equation}
where  
$\delta>0$ and we shall prove that, as $\delta\rightarrow0$, $\delta u_{\delta}$ converges to a constant $\lambda$, 
called "ergodic" constant.
Let us stress that the differential operator in the ergodic problem coincides with the infinitesimal generator $\cal L$ of process~\eqref{process}.

We recall that the study of ergodic problems for equations with periodic, uniformly elliptic, operators has been addressed in~\cite{AL, BLP} while, for periodic, possibly degenerate (still satisfying the 
H\"ormander's condition) operators, we refer the reader to the papers~\cite{AB1, ABM}.

The main difficulties in our problem are the lack of periodicity and the degeneracy of the operator. We shall overcome these issues using some techniques introduced by~\cite{BCM} for  an elliptic operator on the whole space.
Moreover, we shall give an explicit formula for the ergodic constant $\lambda$ in terms of the invariant measure for~\eqref{process}.

In the latter application we 
consider the 
following Cauchy problem:
\begin{equation*}
u_t+{\mathcal L}u=0\quad \textrm{in }(0,+\infty)\times\re^3,\qquad
u(0, x)=f(x)\quad \textrm{on }\re^3,
\end{equation*}
where $\mathcal L$ is the operator defined in (\ref{op_l}).
We will prove that, as $t\to+\infty$, the solution $u$ converges to a constant $\Lambda$ which will be characterised in terms of the invariant measure.

%
%

Finally, we shall show how to extend our previous results to other degenerate operators satisfying H\"ormander condition with possibly unbounded coefficients.

\vskip 2mm

Our future purpose is to 
 use the ergodic problem to study the homogenization problem
\begin{equation}\label{Hom}
-\epsilon\,tr(\sigma(\frac{x}{\epsilon})\sigma^T(\frac{x}{\epsilon})D^2u_{\epsilon})-b(\frac{x}{\epsilon})\cdot Du_{\epsilon}+ f(x,\frac{x}{\epsilon})+au_{\epsilon}=0 \mbox{ in } \re^3,
\end{equation}
where $\sigma$ has the form~\eqref{matrixH}.
 In this case the approximated cell problem formally coincides with the problem (\ref{ergointro}).\\
 For the study of homogenization problems for periodic, possibly nonlinear, degenerate 
 (still satisfying the 
H\"ormander's condition) operators, we refer the reader to the papers~\cite{AB1, BMT1, MS}.

This paper is organized as follows: 
Section \ref{IM} contains the main result of the paper: we find conditions on the drift $b$ such that a Lyapunov functions does exist and by means of this function we prove the existence and uniqueness of an invariant measure associated 
to our process.
In Section \ref{sect:liou}, we establish a Liouville type result assuming the existence of a Lyapunov-like function.
Section \ref{Appl} is devoted  to our applications:
in Section  \ref{sectergodic} we study the ergodic problem through stationary problems with vanishing discount,
while in Section
\ref{large time} we consider the long time behaviour of a Cauchy problem.
In Section \ref{sect:gen_case} we generalise the previous results to a more general class of subelliptic operators, encompassing {\it e.g.} the Grushin one.
The Appendix contains a condition equivalent to (\ref{Lipintr}) which will be useful to manage the Lyapunov function founded in Section~\ref{IM}.

%
%
\section{Existence and uniqueness of the invariant measure}\label{IM}
This section is devoted to the invariant measure for process~\eqref{process}. Let us recall (see \cite{LM} or Proposition~ \ref{thmdelta} below) that, when the matrix associated to the infinitesimal generator 
$\mathcal L$ 
is a strictly definite positive matrix, a sufficient condition for the existence of an invariant measure is given by:\\
there exists a {\em Lyapunov-like} function such that 
\begin{eqnarray}
&& w\in C^{\infty}(B_0^C)\cap C^0(\re^3)\label{C20}\\
&&\mathcal Lw
\geq 1,\ \mbox{ in } B_0^C\nn\\
&&w\geq 0 \mbox{ in } B_0^C ,\ w=0 \mbox{ on } \partial B_0,\nn
\end{eqnarray}
where $B_0$ is a ball centered in $0$ with suitable radius.
(For less regular functions $w$, we refer to (\cite{LM})).

In our case, the matrix $A=\sigma\sigma^T$ in (\ref{matrixA}) is degenerate in any point, and the rank of the matrix is  2. 
In order to overcome this issue, for $\rho>0$, we introduce the approximating operators
\begin{equation}\label{Ldelta}
\mathcal L_{\rho}w:=-tr(A_{\rho}(x)D^2 w)-b(x)Dw,
\end{equation}
where 
\begin{equation}\label{matrixArho}
A_{\rho}(x)=\begin{bmatrix}
	1 &0&2x_2\\
0& 1&-2x_1\\
2x_2& -2x_1& 4(x_1^2+x_2^2)+\rho^2
\end{bmatrix}=\sigma(x)\sigma^T(x)+\begin{bmatrix}
	0 &0&0\\
0& 0&0\\
0&0&\rho^2
\end{bmatrix}.
	\end{equation}
In the following Lemma we collect some useful properties of $\mathcal L_\rho$.
\begin{lemma}\label{strictell}
The matrix $A_{\rho}(x)$ is locally strictly positive definite (namely, for any compact $K\subset\re^3$, there holds $\lambda A_{\rho}(x)\lambda^T\geq \nu(x)|\lambda|^2$ for any $x\in K$, with $\nu(x)\geq a(K,\rho)>0$) and it is positive definite in $\re^3$.

Moreover, there exists a $3\times 3$ matrix $\sigma_{\rho}(x)$ with linear coefficients such that 
\begin{equation}
\label{arhosigma}
A_{\rho}(x)=\sigma_{\rho}(x)\sigma^T_{\rho}(x).
\end{equation}
\end{lemma}
\begin{proof}
Set $\alpha=4(x_1^2+x_2^2)+\rho^2+1$.
The eigenvalues of $A_{\rho}$ are
$$\lambda_1=1,\quad
\lambda_{2,3}=\frac{\alpha\pm\sqrt{\alpha^2-4\rho^2}}{2}.$$
It is easy to remark that $\lambda_2\geq \frac{1}{2}$.\\
The last eigenvalue is $\lambda_{3}= \frac{2\rho^2}{\alpha+\sqrt{\alpha^2-4\rho^2}}>\frac{\rho^2}{\alpha}$
hence, for any fixed $R>0$, if  $x_1^2+x_2^2\leq R^2$, $\alpha\leq 4R^2+\rho^2+1$
and $\lambda_{3}>\frac{\rho^2}{4R^2+\rho^2+1}>0$.\\
The matrix
\begin{equation}\label{matrixsigmarho}
\sigma_{\rho}(x)=\begin{bmatrix}
	1 &0&0\\
0& 1&0\\
2x_2& -2x_1&\rho
\end{bmatrix}.
	\end{equation} 
verifies \eqref{arhosigma} 
 \end{proof}
\begin{remark} \rm{
From (\ref{arhosigma}), 
beside being uniformly elliptic, 
the operator $-\mathcal L_\rho$ is also the infinitesimal generator of the stochastic process 
\begin{equation}\label{processrho}
dX^{\rho}_t=b(X^{\rho}_t)dt +\sqrt2\sigma_{\rho}(X^{\rho}_t)dW_t,
\end{equation}
where $\sigma_\rho$ is defined in~\eqref{matrixsigmarho} and $W_t= (W_{1t}, W_{2t}, W_{3t})$ and $W_{1t}$, $W_{2t}$,   $W_{3t}$ are three independent Brownian motions whereas our starting process~\eqref{process} only contains two independent Brownian motions.}
\end{remark}
Now, we want to prove that, for some classes of drifts $b$, there exists a function $w$ satisfying~\eqref{C20} with $\mathcal L$ replaced by  $\mathcal L_\rho$. To this end, we consider a continuous drift $b=(b_1,b_2,b_3)$ such that
\begin{equation}\label{hp_b}
  b_i(x)=  b_i(x_i),\qquad
\left\{\begin{array}{ll}
b_i(x_i)\leq-\frac{C_i}{|x_i|^{1-\alpha}} &\textrm{for }x_i\geq R\\
b_i(x_i)\geq\frac{C_i}{|x_i|^{1-\alpha}} &\textrm{for }x_i\leq -R
\end{array}\right.
\end{equation}
for some constants $\alpha\geq 0$, $R>0$ and $C_i>0$ ($i=1,2,3$).\\
Note that Lemma \ref{Lya} here below, holds also for $\rho=0$ then we have a Lyapunov-like function $w$ ( i.e. satisfying condition (\ref{C20}))
also for the degenerate starting problem where $\mathcal L$ is given by (\ref{L}).\\
Similar conditions to (\ref{hp_b}) was obtained in \cite{LM} with $\sigma=I$ the identity matrix.
\begin{lemma}\label{Lya}
Assume $\sigma$ as in~\eqref{matrixH}. Assume that $b$ is a continuous function verifying~\eqref{hp_b} with
\begin{itemize}
\item[(i)] either $\alpha>0$,
\item[(ii)] or $\alpha=0$ and sufficiently large $C_i$.
\end{itemize}
Then, there exists a $R_0$ and a $C^\infty$ function~$w$ which satisfies
\begin{equation}\label{C2}
\mathcal L_{\rho}w\geq 1 \quad \textrm{in }\overline{B(0,R_0)}^C,\quad
w\geq 0\quad \textrm{in } \overline{B(0,R_0)}^C,\quad \lim_{\vert x\vert\rightarrow \infty} w=\infty
\end{equation}
for $\rho$ sufficiently small.
\end{lemma}
\begin{proof}
We set 
$$w:=\frac{(x_1^4+x_2^4)}{12}+\frac{x_3^2}{2}.$$
 Then, there holds
$$
\mathcal L_{\rho}w=
-5(x_1^2+x_2^2)-\rho^2-\frac{1}{3}(b_1x_1^3+b_2x_2^3)- b_3x_3.
$$
We denote $K_i:=\max_{x_i\in[-R,R]} \vert b_i(x_i)\vert$.

\noindent{\it Case ($i$)}. Assume $\alpha>0$. We want to prove that there exists $R_0$ such that $\mathcal L_\rho w>1$ in $\overline{B(0,R_0)}^C$ for $\rho$ sufficiently small. To this end, we split the arguments in several cases.

\noindent ($I$). If $|x_i|\geq R$ for any $i\in\{1,2,3\}$, then
$$
\mathcal L_\rho w\geq x_1^2(-5+C_1 |x_1|^\alpha/3)+x_2^2(-5+C_2|x_2|^\alpha/3) +C_3|x_3|^\alpha-\rho^2.
$$
Hence, for $|x_1|, |x_2|>R_1:=\max\{(15/C_1)^{1/\alpha}, (15/C_2)^{1/\alpha}, R\}$, $|x_3|\geq R_3:=\max\{C_3^{-1/\alpha}, R\}$, we get:
$\mathcal L_\rho w\geq 1$ for $\rho$ sufficiently small.

\noindent ($II$). If $|x_1|, |x_2|\leq R_1$ and $|x_3|\geq R$, then
$$
\mathcal L_\rho w\geq -10R_1^2-\rho^2-R^3(K_1+K_2)/3+C_3|x_3|^\alpha
$$
(here, we used the relation: $-b_ix_i^3\geq 0$ for $|x_i|\in[R,R_1]$, $i=1,2$). Hence, for $|x_3|\geq \tilde R_3$ with $\tilde R_3$ sufficiently large, taking $\rho$ sufficiently small, we get
$\mathcal L_\rho w\geq 1$.

\noindent ($III$). If $|x_1|\leq R_1$, $|x_2|\geq R_1$ and $|x_3|\geq R$ (and similarly, for $|x_1|\geq R_1$, $|x_2|\leq R_1$ and $|x_3|\geq R$), then
$$
\mathcal L_\rho w\geq -5R_1^2+x_2^2(-5+|x_2|^\alpha/3)-\rho^2-K_1R^3/3+C_3|x_3|^\alpha.
$$
Hence, for $|x_3|\geq \tilde R_3$, we get $\mathcal L_\rho w\geq 1$ for $\rho$ sufficiently small.

\noindent ($IV$). If  $|x_1|\leq R_1$, $|x_2|>R$, $|x_3|<\tilde R_3$ (and similarly for  $|x_1|> R$, $|x_2|\leq R_1$, $|x_3|<\tilde R_3$), then
$$
\mathcal L_\rho w\geq |x_2|^2(-5+C_2|x_2|^\alpha/3)-5R_1^2-\rho^2-K_1R^2/3-K_3R.
$$
Hence, for $|x_2|>R_1$, we get $\mathcal L_\rho w\geq 1$ for $\rho$ sufficiently small.

In conclusion, gluing together all these cases, we accomplish the proof for $\alpha>0$.

\noindent{\it Case ($ii$)}.
Assume $\alpha=0$; we want to prove that there exist some constants $C_i$ and a radius $R_0$ such that $\mathcal L_{\rho}w\geq 1 $  in $\overline{B(0,R_0)}^C$.

\noindent ($I$). If $|x_i|>R$ for any $i=1,2,3$, then
$\mathcal L_{\rho}w \geq x_1^2(-5+\frac 1 3 C_1)+x_2^2(-5+\frac 1 3 C_1)+C_3-\rho^2$;
hence, for $C_1, C_2>15$,  $C_3>1$, we have $L_{\rho}w>1$ for $\rho$ sufficiently small.

\noindent ($II$). If $|x_i|<R$ for $i=1,2$ and $|x_3|>R$, then
$\mathcal L_{\rho}w \geq -10R^2-R^3(K_1+K_2)/3-\rho^2+C_3$;
hence, for $C_3>10R^2+R^3(K_1+K_2)/3+1$, we have $\mathcal L_{\rho}w>1$ 
for $\rho$ sufficiently small.

\noindent ($III$). If $|x_1|<R$, $|x_2|>R$ and $|x_3|>R$ (and similarly, for $|x_1|\geq R$, $|x_2|\leq R$ and $|x_3|\geq R$), then
$\mathcal L_{\rho}w \geq-5R^2+x_2^2(-5+C_2/3)-R^2K_1/3-\rho^2+C_3$,
hence, for $C_2>15$, $C_3$ sufficiently large and $\rho$ sufficiently small, we have $\mathcal L_{\rho}w>1$.

\noindent ($IV$). If  $|x_1|\leq R$, $|x_2|>R$, $|x_3|<R$ (and similarly for  $|x_1|> R$, $|x_2|\leq R$, $|x_3|<R$), then
$\mathcal L_{\rho}w \geq -5R^2+x_2^2(-5+C_2/3)-K_1R^2-\rho^2-K_3R$; hence, for $C_2>15$, 
$|x_2|$ sufficiently large and $\rho$ sufficiently small, we have $\mathcal L_{\rho}w>1$.
\end{proof}
\begin{remark}{\rm
Stronger sufficient condition on $b_i$ for the existence of a Lyapunov-like function $w$ satisfying condition (\ref{C20}) could be found using 
$\displaystyle w(x):=\log((x_1^2+x_2^2)^2+x_3^2))$.
}
\end{remark}

\begin{remark}\label{OU}
{\rm
The drifts of the Ornstein-Uhlenbeck operator (i.e., $b(x)=-\gamma x$ for $\gamma>0$) satisfy assumption ($i$) of Lemma~\ref{Lya}.
For further properties of this operator we refer the reader to the monograph \cite{LB}.
}
\end{remark}
In the next proposition we will establish the existence of an invariant measure $m_{\rho}$ 
of the approximating process~\eqref{processrho}. 
This measure will be used in the main theorem of this paper when the invariant measure for the process (\ref{process})
will be obtained as the limit of $m_{\rho}$ as $\rho\rightarrow 0$.
\begin{proposition}\label{thmdelta}
Let $\sigma_{\rho}(x)$ defined by (\ref{matrixsigmarho}) and $b(x)$ be a Lipschitz 
function satisfying (\ref{hp_b}) either with $\alpha>0$ or $\alpha=0$ and $C_i$ sufficiently large.
There exists a unique invariant probability measure $m_{\rho}$ on $\re^3$ for the process~\eqref{processrho}.
\end{proposition}
\begin{proof}	
As proved in Lemma \ref{strictell}
the operator $\mathcal L_{\rho}$ is uniformly elliptic in each bounded set (but the ellipticity constant degenerates in the whole $\re^3$).
We adapt some techniques introduced by \cite{LM} (see 
also \cite{C} for similar arguments), by considering approximate problems in domains $O_n$ such that 
$O_n\nearrow \re^3$ if $n\rightarrow +\infty$.

Let us recall (see \cite{Lnt} or Lemma~\ref{LemmaL} in the Appendix) that condition~(\ref{C2}) is equivalent to the following one:
\begin{eqnarray}
&&\textrm{there exists a function $\overline w\in C^{\infty}(\re^3)$ such that }\label{C2barra}\\
&&\mathcal L_{\rho}\overline w+\chi\overline w= \phi\quad \textrm{in } \re^3,\qquad \lim_{|x|\to+\infty} \overline w=\infty\notag
\end{eqnarray}
where $\chi\in C^{\infty}_0$ and $\phi\in C^{\infty}$ are suitable functions such that, $\chi>0$ on $B_0$, $\textrm{supp} \chi =\bar B_0$ ($B_0$ is a suitable open set) and $\lim_{|x|\to\infty}\phi=\infty$.
(As a matter of facts, this condition is satisfied by the function~$w$ chosen in the proof of~Lemma~\ref{Lya}-($i$)).\\
We  define
$O_n:=\{x\in\re^3 |\ \overline w(x)<M_n\}$ where $M_n\rightarrow +\infty$ if $n\rightarrow +\infty$
and $M_n$ is not a critical value of $\overline w$. Since $ \overline w\rightarrow +\infty$ if $x\rightarrow +\infty$
then $O_n$ are bounded and smooth and $O_n\nearrow \re^3$.\\
Fix  $\rho>0$ and $n$, the results by Bensoussan~\cite[Section 4]{B1} ensure that there exists an unique invariant measure $m^n_{\rho}$ associated to the diffusion process $X_t^{\rho}$ in $O_n$ with reflecting boundary whose infinitesimal generator is $L_{\rho}$ in $O_n$ with boundary conditions
$$\sum_{i,j}{(a_{\rho})}_{ij}\frac{\partial u}{\partial \nu_{j}}=0\qquad \textrm {on }\partial O_n$$
where $\nu$ denotes the unit outward normal to $\partial O_n$ and 
the matrix $A_{\rho}={(a_{\rho})}_{ij}=\sigma_{\rho}\sigma_{\rho}^T$ as in Lemma \ref{strictell}.\\
The invariant measure $m_{\rho}^n$ satisfies the problem
\begin{eqnarray}\label{mdeltan}
&\ds\mathcal L_{\rho}^*m^n_{\rho}:=-\sum_{i,j}\frac{\partial^2 ({(a_{\rho})}_{ij}m^n_{\rho})}{\partial x_i\partial x_j}+ \sum_{i}\frac{\partial(b_im^n_{\rho})}{\partial x_i}=0\qquad \textrm{in }O_n,\\
&\ds\sum_{ij}\nu_i({\frac{\partial ({(a_{\rho})}_{ij}m^n_{\rho})}{\partial x_j}-
b_im^n_{\rho})}=0\qquad \textrm{on } \partial O_n\label{reflec}\\
&\ds\int_{O_n} m^n_{\rho}=1,\quad  m^n_{\rho}>0.\nn
\end{eqnarray}
We have to prove that, as $n\rightarrow +\infty$, $m_{\rho}$ converges  in some sense to $m_{\rho}$ invariant measure to the process with generator $\mathcal L_{\rho}$, i.e. $m_{\rho}$ solves
\begin{eqnarray}\label{mdelta}
&\ds\mathcal L^*_{\rho}m_{\rho}:=-\sum_{i,j}\frac{\partial^2 ({(a_{\rho})}_{ij}m_{\rho})}{\partial x_i\partial x_j}+ \sum_{i}\frac{\partial (b_im_{\rho})}{\partial x_i}=0\qquad\textrm{in }\re^3\\
&\ds\int_{\re^3} m_{\rho}=1,\quad  m_{\rho}\geq 0.\nn
\end{eqnarray}
From Prohorov Theorem and the fact that $\int_{O^n} m^n_{\rho}=1$ 
we know that $m^n_{\rho}\rightharpoonup m_{\rho}$ as $n\rightarrow +\infty$ (possibly passing to a subsequence).\\
We prove now that $\int_{\re^3} m_{\rho}=1$.
Multiplying equation (\ref{mdeltan}) by $\overline w$ defined in \eqref{C2barra}, integrating on $O^n$ and taking into account (\ref{reflec})
we obtain
$$0= \int_{O^n}\mathcal L_{\rho}^{*}m^n_{\rho}\overline w\,=
\int_{O^n}m^n_{\rho}\mathcal L_{\rho}\overline w +
\int_{\partial O^n}m^n_{\rho}
\sum_{i,j}{(a_{\rho})}_{ij}\frac{\partial \overline w}{\partial x_i}\nu_j.$$
Since $\overline w=M_n$ on $\partial O^n$ and $\overline w<M_n$ on $O^n$,
we have
$\frac{\partial \overline w}{\partial x_i}= \frac{\partial w}{\partial \nu}\nu_i$ and 
$\frac{\partial \overline w}{\partial \nu}\geq 0$ on $\partial O^n$.
Then, there holds
$$0=\int_{O^n}m^n_{\rho}\mathcal L_{\rho} \overline w +
\frac{\partial \overline w}{\partial \nu}\int_{\partial O^n}m^n_{\rho}
\sum_{i,j}{(a_{\rho})}_{ij}\nu_i\nu_j,$$
and since
$\ds\sum_{i,j}{(a_{\rho})}_{ij}\nu_i\nu_j\geq0,$
we obtain 
$\int_{O^n}m^n_{\rho}\mathcal L_{\rho} \overline w \leq 0$.
Hence
$$\int_{O^n}m^n_{\rho}\mathcal L_{\rho}\overline w 
=
\int_{O^n}(\phi-\chi \overline w) m^n_{\rho}\leq 0,$$
and
\[
\int_{O^n}\phi m^n_{\rho}\leq \int_{supp\chi}\chi \overline w m^n_{\rho}\leq C
\]
where $C$ is a positive constant independent of $n$.
Let us extend $m^n_{\rho}$ by zero outside $O^n$, and call it again $m^n_{\rho}$, then\\
\begin{equation}\label{stimamndelta}
\int_{\R^3}\phi m^n_{\rho}\leq C
\end{equation}
where $C$ is a positive constant independent of $n$.


Since $\displaystyle\lim_{|x|\rightarrow +\infty}\phi(x)=+\infty$, for any $N$ there exists a $R_N$ such that $\phi(x)>N$ on
$B_{R_N}^C$.
Hence, from (\ref{stimamndelta})
\begin{equation}\label{stimaesternaN}
\int_{B_{R_N}^C}m^n_{\rho}\,\leq \frac{C}{N}.
\end{equation}
Since $\ds\int_{\R^3}m^n_{\rho}=1$
then from (\ref{stimaesternaN})
$$\int_{B_{R_{N}}}m^n_{\rho}\, \geq 1-\frac{C}{N}$$
and from the weak convergence of $m^n_{\rho}$ to $m_{\rho}$, we have
$$\int_{B_{R_N}}m_{\rho}\, \geq 1-\frac{C}{N},$$
hence letting $N\rightarrow +\infty$
we obtain that $\int_{\re^3} m_{\rho}=1$.\\
Moreover from the local regularity $W^{2,p}$ for any $p>1$ of $m^n_{\rho}$ 
since $m^n_{\rho}$ solves equation (\ref{mdeltan}), passing to the limit we easily obtain that
$m_{\rho}$
solves equation (\ref{mdelta}). 

%

\end{proof}

 \begin{remark}{\rm
The condition of strict ellipticity in the compact subsets of~$\re^N$ is sufficient to deduce from (\ref{C2}) the existence of the invariant measure~$m_{\rho}$ (see \cite[Theorem IV.4.1]{Ha} under their assumption B.1). Nevertheless we gave the proof of Proposition~\ref{thmdelta} because it is purely analytic and for the sake of completeness.}
\end{remark}

Now we want to prove that $m_{\rho}$ converges in some sense to $m$, invariant measure to 
the process (\ref{process}), solving
\begin{equation}\label{eqm}
{\cal L}^*m=0,\qquad \int_{\re^3} m\,=1\qquad\textrm{and}\qquad m\geq0.
\end{equation}

\begin{theorem}\label{m}
Let $\sigma$ be defined by (\ref{matrixH}) and $b(y)$ be a continuous 
function satisfying (\ref{hp_b}) either with $\alpha>0$ or $\alpha=0$ and $C_i$ sufficiently large. Then
there exists a unique invariant probability measure $m$ on $\re^3$ for the process~\eqref{process}.%

\end{theorem}
\begin{proof}
The uniqueness of the measure $m$ comes from the results of Arnold, Klieman \cite{AK}, or Ichihara, Kunita
\cite{IK}.

The existence of the invariant measure it is obtained proving that the invariant measure $m_{\rho}$ of 
Proposition \ref{thmdelta} converges, if $\rho$ tends to $0$,  to the measure $m$ associated to the process (\ref{process}).
We proceed analogously to Proposition \ref{thmdelta}.
The measure $m_{\rho}$ satisfies the following conditions:
\begin{equation}\label{mdelta1}
\mathcal L_{\rho}^{*}m_{\rho}=0\qquad\textrm{in }\re^3,\qquad \int_{\re^N} m_{\rho}=1,\qquad  m_{\rho}\geq 0.
\end{equation}
We know that 
$m_{\rho}\rightharpoonup m$ as $\rho\rightarrow 0$ (at least for a subsequence) where 
$m$ is a measure.
We have to prove that $m$ is an invariant measure to 
the process (\ref{process}) i.e. that $m$ solves (\ref{eqm}).\\ 
From condition (\ref{C2}) and the equivalent conditions (\ref{C2barra}), we know that
there exists smooth functions $\chi$ and $\phi$ such that $\overline w$ satisfies 
$\mathcal L_{\rho}\overline w+\chi\overline w= \phi$, in $\re^3$, $\overline w$ and $\phi$ such that
$\rightarrow +\infty$ if $|x|\rightarrow +\infty$ and $\chi$ has compact support.

Multiplying equation (\ref{mdelta}) by such $\overline w$ and integrating on $\re^3$ we obtain
$$0= \int_{\re^3}\mathcal L_{\rho}^{*}m_{\rho}\overline w\,=
\int_{\re^3}\mathcal L_{\rho}\overline w m_{\rho}=
\int_{\re^3}(\phi-\chi\overline w) m_{\rho},$$
hence
\begin{equation}\label{stimamdelta}
\int_{\re^3}\phi\,m_{\rho}= \int_{supp\,\chi}\,\chi\,\overline w\,m_{\rho}\leq C,
\end{equation}
where $C$ is a positive constant independent of $\rho$.
From (\ref{stimamdelta}),
since
$$1=\int_{B_{R_N}}m_{\rho}\,+ \int_{B_{R_N}^C}m_{\rho}\,,$$
then
$$\int_{B_{R_N}}m_{\rho}\, \geq 1-\frac{C}{N}$$
and from the convergence of $m_{\rho}$
$$\int_{B_{R_N}}m\, \geq 1-\frac{C}{N},$$
hence letting $N\rightarrow +\infty$
we obtain $
\int_{\re^3}m=1$.\\
To prove that
$\mathcal L^{*}m=0$ we write, for any $\psi$ smooth, 
$$
0=\int_{\re^3}\mathcal L_{\rho}^{*}m_{\rho}\psi\,=
\int_{\re^3}\mathcal L_{\rho}\psi m_{\rho}\rightarrow 
\int_{\re^3}\mathcal L\psi m= \int_{\re^3}\mathcal L^{*}m\psi\,.$$
Taking account that 
$\mathcal L_{\rho}\psi \rightarrow \mathcal L\psi$ strongly and 
$m_{\rho} \rightharpoonup m$ weakly in $L^1$.
\end{proof}

%
%
\section{A Liouville type result}\label{sect:liou}
In this section, we establish a Liouville type result, which holds true not only in the Heisenberg setting but also for $\sigma$ whose columns satisfy the general H\"ormander condition.
This result will be stated in Proposition \ref{Liouville}.
Although 
in the proof of Theorems \ref{conv1} 
and \ref{longtime} below
it will be applied to the particular case of a regular solution,  Proposition \ref{Liouville} contains a general 
statement which has its own independent  interest.

Let us first recall from~\cite{Ho} the definition of H\"ormander condition.

\begin{definition}\label{Hor}
The vector fields 
$ X_j$,  $j=1,\dots m$, satisfy the H\"ormander condition if 
 $X_1,\dots X_m$ and their commutators of any order span $\re^N$ at each point of $\re^N$.
\end{definition}
\begin{proposition}\label{Liouville}
Consider the problem
\begin{equation}\label{L}
\mathcal LV=-tr(\sigma(x)\sigma^T(x)D^2V)-b(x)\cdot DV=0,\quad x\in\re^N
\end{equation}
where $\sigma$ and $b$ are smooth functions
and the vector fields
$ X_j=\sigma^j\cdot\nabla,\,  j=1,\dots m$ satisfy the H\"ormander condition as in definition (\ref{Hor}).
Assume that there exists $w(x)\in C^{\infty}(\re^N)$ and $R_0>0$ such that 
\begin{equation}\label{lyap}
\mathcal L w \geq 0\quad \textrm{in }  \overline{B(0,R_0)}^C,\qquad
w(x)\rightarrow +\infty\quad \textrm{as } |x|\rightarrow +\infty.
\end{equation}
Then:
\begin{itemize}
\item[(i)] every viscosity subsolution $V\in USC(\re^N)$ to (\ref{L}) such that $\displaystyle \limsup_{|x|\rightarrow +\infty}\frac{V}{w}\leq 0$
is constant;
\item[(ii)] every viscosity supersolution $V\in LSC(\re^N)$ to (\ref{L}) such that $\displaystyle \liminf_{|x|\rightarrow +\infty}\frac{V}{w}\geq 0$ is constant.
\end{itemize}
\end{proposition}
\begin{proof}
The proof uses the same arguments as in~\cite{LM} (see also~\cite[Lemma 4.1 and remark 4.1]{BCM}). For the sake of completeness, we shall give the proof of case ($i$); being similar, the proof of case ($ii$) is omitted.

Let us first observe that if $\psi\in\mathcal C^2(A)$ ($A$ is any open set $A\subset \overline{B(0,R_0)}^C$) is a classical supersolution in A, i.e.
\[\mathcal L \psi \geq 0 \mbox{ in } A
\]
then $w+\psi$ is a viscosity supersolution  in A, i.e.
\[\mathcal L (w+\psi) \geq 0 \mbox{ in } A.
\]
Define for each $\eta>0$:
$$V_{\eta}:=V(x)-\eta w(x).$$
We claim that $V_{\eta}$  is a viscosity subsolution in $\overline{B(0,R_0)}^C$ i.e.
\begin{equation}\label{cl:liou}
\mathcal L(V_{\eta})\leq 0  \textrm{ in } \overline{B(0,R_0)}^C.
\end{equation}
Indeed, let us assume by contradiction that there exists $\psi\in\mathcal C^2(\overline{B(0,R_0)}^C)$  such that  $V_{\eta}-\psi$ attains a strict maximum in some point~$\overline x\in \overline{B(0,R_0)}^C$,  
$V(\overline x)=\eta w(\overline x)+\psi(\overline x)$, 
and that there holds
$$\mathcal L(\psi)(\overline x)> 0.$$
By the the continuity of the coefficients of $L$, and the regularity of $\psi$ there exists a $r_0>0$
 such that
\begin{equation}\label{psi}
\mathcal L(\psi)(x)> 0\textrm{ in }\ B(\overline x, r_0)\subset \overline{B(0,R_0)}^C.
\end{equation}
As remarked above 
$\eta w+\psi$ is a supersolution in $B(\overline x, r_0)$.
Moreover there exists $\alpha>0$ such that $V(x)< \eta w(x)+\psi(x)-\alpha$ for any $x\in \partial B(\overline x, r_0)$.
Then by a local comparison principle (see \cite{BDL} or \cite{Bo} for classical solutions), $V(x)\leq \eta w(x)+\psi(x)-\alpha$ in $B(\overline x, r_0)$ and for $x=\overline x$ we get a contradiction and our claim~\eqref{cl:liou} is proved.

Thanks to $V_{\eta}\rightarrow -\infty$ as $|x|\rightarrow +\infty$, there exist $R_1(\eta)=R_1>R_0$ such that
$$V_{\eta}(x)\leq \sup_{|z|=R_0}V_{\eta}(z),\quad \forall |x|\geq R_1$$
then, using the weak maximum principle applied to $V_{\eta}$, 
\[
\max_{B(0,R_1)\backslash \overline{B(0,R_0)}}V_{\eta}=\max_{\partial B(0,R_0)}V_{\eta}
\]
and this implies that
\[
V_{\eta}(x)\leq\max_{\partial B(0,R_0)}V_{\eta}, \quad \forall x\in \overline{B(0,R_0)}^C.
\]
Letting $\eta\rightarrow 0$ in the preceding inequality:
\[
V(x)\leq\max_{\partial B(0,R_0)}V, \quad \forall x\in \overline{B(0,R_0)}^C.
\]
Therefore $V$ attains its global maximum so it is a constant by the strong maximum principle established by Bardi and Da Lio~\cite[Corollary 3.2]{BDL}. 
\end{proof}
\begin{remark}{\rm
Note that, in the Heisenberg group, the function $w$ introduced in Lemma
\ref{Lya} satisfies assumptions \eqref{lyap}. }
\end{remark}
\begin{remark}{\rm
Let us stress that the above arguments work for any linear operator $\mathcal L$ with continuous coefficients, satisfying a local comparison principle and a strong maximum principle.}
\end{remark}
\begin{remark}{\rm 
Note that conditions on the sub and super solutions in i) and ii) imply the boundedness of the sub and super solutions.}
\end{remark}
\section{Applications}\label{Appl}
In this section we provide two applications of the previous results. In both cases we will use the existence of the invariant measure for the process (\ref{process}) proved in Section \ref{IM}  and the Liouville type property obtained in Section \ref{sect:liou}.
Summarizing, we shall prove that
$$
\lim_{\delta\rightarrow 0}\delta u_{\delta}(x)=  
\lim_{t\to +\infty}u(t, x)=
\lim_{t\to +\infty}\frac{v(t,x)}{t}=
\int_{\re^3}f(x)dm(x),$$
where $m$ is the invariant measure of Section \ref{IM} and $u_{\delta}$, $u$ and $w$ are the solutions respectively of
$$\delta u_{\delta}(x)+{\mathcal L}u=f(x),  \qquad\textrm{in }\re^3,$$
$$
u_t+{\mathcal L}u=0,\quad u(0,x)=f(x),  \qquad \textrm{in }(0,+\infty)\times\re^3,$$
$$
v_t+{\mathcal L}v=f,\quad
v(0,x)=0,  \qquad \textrm{in }(0,+\infty)\times\re^3,$$
and ${\mathcal L}$ is the infinitesimal generator of the process (\ref{process}), i.e. the operator defined in (\ref{op_l}).

\subsection{The ergodic problem}\label{sectergodic}
In this section we tackle the following ergodic problem. We consider the family of problems 
\begin{equation}\label{probcellappross}
\delta u_{\delta}(x)-tr(\sigma(x)\sigma^T(x)D^2 u_{\delta})-b(x)Du_{\delta}=f(x)\qquad \textrm{in }\re^3,
\end{equation}
where  
$\delta>0$
and we investigate about the convergence as $\delta\rightarrow 0$ of $\delta u_\delta$ to a constant $\lambda$ called the ergodic constant.
Throughout this section, we assume
\begin{itemize}
\item[($A_1$)] $\sigma$ is defined in ~\eqref{matrixH}
\item[($A_2$)] $b\in C^\infty(\re^3)$ and satisfies the hypotheses of Lemma~\ref{Lya}
\item[($A_3$)] $f\in C^\infty(\re^3)\cap L^\infty(\re^3)$ 
\end{itemize}
The next two Lemma contain several properties of $u_\delta$ which will be used later on.
\begin{lemma}\label{exudelta0}
Under Assumptions~$(A_1)$-$(A_3)$, there exists an unique smooth viscosity solution $u_{\delta}$ of the approximating problem 
(\ref{probcellappross}) such that
\begin{equation}\label{sublin}
|u_{\delta}(x)|\leq \frac C\delta, \qquad \forall x\in \re^3,
\end{equation}
for some positive constant~$C$ independent of $\delta$.
\end{lemma}
\begin{proof}
The uniqueness follows from the comparison principle proved in \cite{Bo}.
By assumption~($A_3$) it is easy to see that
$w^{\pm}=\pm\frac{C}{\delta}$ with $C$ sufficiently large is respectively a 
 supersolution and a subsolution for problem 
(\ref{probcellappross}).
In conclusion, applying Perron's method, we infer the existence of a solution to~\eqref{probcellappross} verifying~(\ref{sublin}).
\end{proof}
\begin{lemma}\label{lemmaholdervdelta}
Under assumptions ($A_1$)-($A_3$), the functions $v_{\delta}:=\delta u_{\delta}$, where $u_{\delta}$ 
is the solution of problem
(\ref{probcellappross}), are locally uniformly H\"older continuous. Namely, there exists $\alpha\in(0,1)$ such that for every compact $K\subset \re^3$ there exists a constant $N$ such that 
\begin{equation}\label{holdervdelta}
|v_{\delta}(x_1)-v_{\delta}(x_2)|\leq N|x_1-x_2|^{\alpha}, \quad \forall x_1,x_2\in K,\ \forall \delta\in(0,1).
\end{equation}
The constant $N$ only depends on $K$ and on the data of the problem (in particular is independent of $\delta$).
\end{lemma}
\begin{proof}
The statement is a direct consequence of the result of Krylov \cite{Kri}.
For the sake of completeness let us sketch how to apply Krylov's result to our case.
From Lemma \ref{exudelta0} the function $v_{\delta}$ is uniformly bounded and smooth and solves the following equation
\begin{equation}\label{probcellapprossv}
\delta v_{\delta}-tr(\sigma(x)\sigma^T(x)D^2 v_{\delta})-b(x)Dv_{\delta}=\delta f(x)\qquad\textrm{in }\re^3.
\end{equation}
We observe that equation (\ref{probcellapprossv}) can be written in the form 
\begin{equation}\label{probcellapprossvKry}
-L_0v_{\delta}+v_{\delta}:=-\sigma^{ik}\partial_{x_i}(\sigma^{jk}\partial_{x_j}v_{\delta})-BDv_{\delta}+v_{\delta}=
\delta f+(1-\delta)v_{\delta}
\end{equation}
where  $B_j= b_j-\sum_{jk}\sigma^{ik}\partial_{x_i}\sigma^{jk}$ and 
$L_0=\sigma^{ik}\partial_{x_i}(\sigma^{jk}\partial_{x_j}\cdot)$.\\
For $\delta$ fixed, consider the problem 
\begin{equation*}
\left\{\begin{array}{ll}
-L_0v_{\delta, n}+v_{\delta, n}= \delta f+(1-\delta)v_{\delta}&\quad\textrm{in }B(0,n)\\
v_{\delta, n}=0&\quad \textrm{on }\partial B(0,n).
\end{array}\right.
\end{equation*}
We observe that $\{v_{\delta, n}\}$ is a equibounded family (by the same arguments of Lemma 
\ref{exudelta0}). 
\cite[Theorem 2.1]{Kri}
of Krylov ensures that there exists $\alpha\in(0,1)$ such that for every compact $K\subset \re^3$ there exists a constant $N_1$ 
(independent of $\delta$, $n$) such that 
\begin{equation}\label{vdeltan}
|v_{\delta, n}(x_1)-v_{\delta, n}(x_2)|\leq N_1|x_1-x_2|^{\alpha}, \quad \forall x_1,x_2\in K.
\end{equation}
By Ascoli-Arzel\`a Theorem, letting $n\rightarrow +\infty$ (possibly passing to a subsequence) we get that $v_{\delta, n}$ converges locally uniformly to a function $V_{\delta}$. By the stability and uniqueness results we 
infer $V_{\delta}=v_{\delta}$. Moreover, passing to the limit in $n$ in (\ref{vdeltan}), we get
(\ref{holdervdelta}).
\end{proof}
In the next result we prove that $\delta u_{\delta}$ converges to a constant which will be characterize in terms of the invariant measure of the process~\eqref{process}.
\begin{theorem}\label{conv1}
Under assumptions ($A_1$)-($A_3$), 
the solution $u_{\delta}$ of  problem (\ref{probcellappross}) given in Lemma~\ref{exudelta0} satisfies
\begin{eqnarray}\label{4.10}
&&\lim_{\delta\rightarrow 0}\delta u_{\delta}=
\int_{\re^3}f(x)dm(x),\ \textrm{locally uniformly},
\end{eqnarray}
where $m$ is the invariant measure of process (\ref{process}) founded in Section \ref{IM}.
\end{theorem}
\begin{proof}
We shall proceed following some arguments of~\cite{BCM}. 
The functions
$v_{\delta}:=\delta u_{\delta}$ solve (\ref{probcellapprossv})
and,
from estimate (\ref{sublin}),
satisfy 
\begin{equation}\label{equilin}
|v_{\delta}|\leq C,\qquad\textrm{in }\re^3,
\end{equation}
with $C$ independent of $\delta$, hence they are
uniformly bounded in $\re^3$.
From Lemma \ref{lemmaholdervdelta}
$v_{\delta}$
are also uniformly H\"older continuous in any compact set of $\re^3$.
Then by the Ascoli-Arzel\`a theorem there is a sequence $\delta_n\rightarrow 0$ and a continuous function $w$ such that 
$v_{\delta_{n}}\rightarrow v$ locally uniformly; by stability, $v$ is a solution of
\begin{equation}\label{probcellav}
-tr(\sigma(x)\sigma^T(x)D^2 v)-b(x)Dv=0,\ x\in\re^3,
\end{equation}
hence $v\in C^{\infty}$ by the hypoellipticity of the operator (see \cite{Bo}).
Then by Proposition \ref{Liouville}, $v$ is constant.

In conclusion, we have that, possibly passing to a subsequence, $\{\delta u_\delta\}_\delta$ converges locally uniformly to a constant. Now, it remains to prove that this constant is independent of the subsequence chosen and that is has the form~\eqref{4.10}. By standard arguments of optimal control theory (see~\cite{FS}), the function $u_\delta$ can be written as
\begin{equation*}
u_\delta (x)=\mathbb E_x\int_0^{+\infty}f(X_t)e^{-\delta t}\, dt
\end{equation*}
where $X_t$ is the process in~\eqref{process} with initial data $X_0=x$ while $\mathbb E$ denotes the expectation.
Integrating both sides with respect to the invariant measure, we infer
\begin{eqnarray*}
\int_{\re^3}u_\delta(x)\, dm(x) &=&
\int_0^{+\infty}\left(\mathbb E_x\int_{\re^3}f(X_t)\, dm(x)\right)e^{-\delta t}\, dt\\
&=&\int_0^{+\infty}\left(\int_{\re^3}f(x)\, dm(x)\right)e^{-\delta t}\, dt\\
&=&\frac1\delta \int_{\re^3}f(x)\, dm(x)
\end{eqnarray*}
where the second inequality is due to the definition of invariant measure. Taking into account that every convergent subsequence of~$\{\delta u_\delta\}_\delta$ must converge to a constant, we conclude that all the sequence~$\{\delta u_\delta\}_\delta$ converges to $\int_{\re^3}f\,dm$. 
\end{proof}

\subsection{Large time behavior of solutions}\label{large time}
This section concerns the asymptotic behavior for large times of the solution of the parabolic Cauchy problem:
\begin{equation}\label{u}
\left\{\begin{array}{ll}
u_t+{\mathcal L}u=0&\quad \textrm{in }(0,+\infty)\times\re^3,\\
u(0,x)=f(x)&\quad \textrm{on }\re^3,
\end{array}\right.
\end{equation}
where $\mathcal L$ is the operator defined in (\ref{op_l}).
Let us recall that, 
for periodic fully nonlinear equations, 
this issue was studied in \cite[Theorem 4.2]{AB1}. We quote here also the results in 
the manuscript~\cite{LM}.
\begin{theorem}\label{longtime}
 Under the assumptions of Theorem 2.1 and Proposition 3.1, 
for $f(x)\in C^0(\re^3)\cap L^{\infty}(\re^3)$, 
the solution $u(t, x)$ of problem (\ref{u}) verifies
$$
\lim_{t\to +\infty}u(t, x)= \int_{\re^3} f(x)\,dm(x), \qquad \textrm{locally uniformly in } x, $$
where $m$ is the invariant measure of process (\ref{process}) founded in Section \ref{IM}.
\end{theorem}
\begin{proof}
Since $\pm \|f\|_{\infty}$ are 
sub and supersolution of (\ref{u}), by the comparison principle 
we have that
\begin{equation}\label{ubound}
\|u\|_{\infty}\leq\|f\|_{\infty}.
\end{equation}
Arguing as in 
\cite[Theorem 4.2]{AB1} we get, for some $c>0$,
$
|u(t+s,x)-u(t,x)|\leq cs,$ and in particular $|u_t(t,x)|\leq c$.
Moreover classical results on regularity of subelliptic operators give that $u(t, \cdot)$ are locally H\"older continuous on $x$ uniformly in $t$ (see \cite{Bo,Kri}).
Hence by Ascoli-Arzel\`a theorem for any sequence $t_n\rightarrow +\infty$ there exists a subsequence $t_{n_k}$
such that  $u(t_{n_k}, \cdot)\rightarrow v$ locally uniformly for some $v\in C^0(\re^3)$.
By standard arguments (see \cite[Theorem 4.2]{AB1}),
$v$ is the solution
of ${\mathcal L}v=0$; hence, by Proposition 3.1 (the Liouville type result), it is a constant. Therefore, we have 
\begin{equation}\label{convC}
u(t_{n_k}, \cdot )\rightarrow C,\ \textrm{ locally uniformly }.
\end{equation}
We show now that the constant $C$ is independent of the chosen sequence.
Let us consider an arbitrary sequence $s_{n}$ such that $s_{n}\rightarrow +\infty$ and $u(s_{n}, \cdot )\rightarrow K$ locally uniformly.
From (\ref{4.12}) 
$$\int_{\re^3} u(s_{n}, x)\, dm(x) = \int_{\re^3} f(x)\, dm(x)$$
Using \eqref{ubound},  $\int_{\re^3}  dm(x) =1$ and the dominated convergence theorem 
$$K= \int_{\re^3} f(x)\, dm(x).$$
\end{proof}
\begin{remark}
{\rm Let us consider the following Cauchy problems
\begin{equation}\label{w}
\left\{\begin{array}{ll}
v_t+{\mathcal L}v=f&\quad \textrm{in }(0,+\infty)\times \re^3\\
v(0,x)=0&\quad \textrm{on }\re^3,
\end{array}\right.
\end{equation}
where $\mathcal L$ is the operator defined in (\ref{op_l}) and $f$ is a function as in Theorem~\ref{longtime}.

By means of the Duhamel formula and a change of variables
the solution 
$v$ can be written as
$v(t,x)= \int^t_0u(\tau,x)d\tau$ where
$u$ is the solution of (\ref{u}).
Hence
the statement of Theorem~\ref{longtime} can be rephrased as
$$\lim_{t\to +\infty}\frac{v(t,x)}{t}= \int_{\re^3} f(x)\,dm(x).$$
}
\end{remark}

\section{The general case}\label{sect:gen_case}

In this section we address to the process (\ref{process}) under the following assumptions:
\begin{equation}\label{sigmagen}
\left\{\begin{array}{l}
\ \sigma(x)\in C^{\infty}(\re^N), \\
\|\sigma(x)\|\leq C(|x|+1),\qquad \textrm{for some } C>0;\\
\textrm{the columns of $\sigma$  satisfy H\"ormander condition}.
\end{array}\right.
\end{equation}
\begin{eqnarray}\label{bfgen}
b(x), f(x)\in C^{\infty}(\re^N),\ \|b(x)\|,\ |f(x)|\leq C(|x|+1),\ C>0.
\end{eqnarray}
\begin{equation}\label{arogen}
\left\{\begin{array}{l}
\textrm{For $A:=\sigma\sigma^T$, there exists $\{A_{\rho}(x)\}_{\rho\in(0,1)}$ with  $A_{\rho}=\sigma_{\rho}\sigma^T_{\rho}$},\\ \sigma_{\rho}\in C^{\infty}(\re^N),\
A_{\rho}\rightarrow A \textrm{ in } L^{\infty} \textrm{ and } A_{\rho} \textrm{ is locally definite positive}.
\end{array}\right.
\end{equation}
\begin{equation}\label{liapgen}
\mbox{There exists a function $w$ which verifies \eqref{C2} for any $\rho$ sufficiently small.}
\end{equation}
The grow assumptions on $\sigma$ in (\ref{sigmagen}) and on $b$ and $f$ in (\ref{bfgen}) allow us to obtain the existence of a process $X_t$ in (\ref{process}).\\
 Under assumptions (\ref{sigmagen}),  (\ref{bfgen}), the Liouville type result contained in Proposition \ref{Liouville} 
 still holds true. In fact 
 the results of 
Bony \cite{Bo} on comparison principle and strong maximum principle hold also in this setting if we
observe that
$$ -tr(\sigma\sigma^TD^2u)=\sum_jX_j^2u- C(x)\cdot Du,$$
where $C(x)=D\sigma^j\cdot \sigma^j$ and $\sigma^j$ are the columns of the matrix $\sigma$. 
  \begin{theorem}\label{mgen}
Under assumptions (\ref{sigmagen})-(\ref{liapgen}) there exists an invariant probability measure $m$ associated to the diffusion process (\ref{process}).
\end{theorem}
\begin{proof}
We observe that, by assumptions (\ref{arogen}), (\ref{liapgen}), 
there exists an unique invariant measure $m_{\rho}$ for the process with diffusion $\sigma_{\rho}$. Then,
arguing as in the proof of Theorem \ref{m}, using again the function $w$ in (\ref{liapgen}) we obtain the existence of the invariant measure
associated to
the process (\ref{process}).
\end{proof}

\begin{corollary}
Under assumptions (\ref{sigmagen})-(\ref{liapgen}) Theorems \ref{conv1} and \ref{longtime} hold true.
\end{corollary}
\begin{example}The Grushin operator\\
\rm{ 
For $x=(x_1, x_2)\in\re^2$, consider the diffusion matrix 
\begin{equation}\label{sigma}
\sigma(x)=\left(\begin{array}{cc} 1&0\\0&x_1\end{array}\right)
\end{equation}
and observe that
 $\sigma$ satisfies (\ref{sigmagen}) since $X_1=(1,0)$ and $[X_1,X_2]=(0,1)$ span all $\re^2$.
In this case the infinitesimal generator is 
$$\mathcal LV= -V_{x_1x_1}-x_1^2V_{x_2x_2}-b(x)DV.$$

We take $f$ and $b(x)=(b_1(x_1), b_2(x_2))$ satisfying (\ref{bfgen}) with 
\begin{equation}\label{bgrushin}
\left\{  \begin{array}[c]{ll}
- b_1x_1\geq 6,   &\mbox{ if } \vert x_1\vert>1,  \\
- b_2x_2\geq 1,   &\mbox{ if } \vert x_2\vert>1,\\
 b_i \mbox{  are bounded in } [-1,1],& i=1,2.
  \end{array}\right.
\end{equation}
Under these assumptions it is easy to check that the matrix 
$$
\sigma_{\rho}(x)=\left(\begin{array}{ccc} 1&0&0\\0&x_1&\rho\end{array}\right)
$$
satisfies (\ref{arogen}) and that the function $W(x)=\frac{1}{12}x_1^4+\frac{1}{2}x_2^2$ satisfies 
(\ref{liapgen}).
In conclusion, since all the hypotheses (\ref{sigmagen})-(\ref{liapgen}) are satisfied,
Theorem \ref{mgen} 
apply.}
\end{example}
\begin{remark}\rm {
Lions-Musiela in \cite{LM} have considered a similar degenerate case
but in their paper the elements of the matrix are bounded in $\re^2$ in this way
\begin{equation}\label{sigmaLM}
\sigma(x)=\left(\begin{array}{cc} 1&0\\0&\frac{x_1}{\sqrt{1+x_1^2}}\end{array}\right).
\end{equation}  }
\end{remark}

\appendix
\section{Appendix}\label{sect:appendix}
In the following Lemma we state the equivalence between conditions~\eqref{C2} and~\eqref{C2barra}. This property has already been established by P.L. Lions~\cite{Lnt}; however, for the sake of completeness we shall provide the proof.
\begin{lemma}\label{LemmaL}
Consider a linear operator 
\[{\cal G}(u):=-tr (\tau \tau^T D^2 u)-\beta\cdot Du\]
where $\tau $ is a matrix whose columns verify the H\"ormander condition (\ref{Hor}), $\tau $ and $\beta$ are smooth functions with 
\[|\tau(x)|, |\beta (x)|\leq C(1+|x|)\qquad \forall x\in\re^N.\]
Then, conditions~\eqref{C2} and~\eqref{C2barra} are equivalent; namely the following properties are equivalent:
\begin{itemize}
\item[$(i)$] there exists $w\in C^\infty(\re^N)$ such that
\[{\cal G}(w)\geq 1 \quad \textrm{in }\overline{B(0,R_0)}^C,\quad
w\geq 0 \quad \textrm{in }\overline{B(0,R_0)}^C,\quad \lim_{|x|\to+\infty}w=+\infty\]
for some constant~$R_0>0$;
\item[$(ii)$] there exists $\bar w\in C^\infty(\re^N)$ such that
\[{\cal G}(\bar w)+\chi\bar w=\phi \quad \textrm{in }\re^N,\qquad
 \lim_{|x|\to+\infty}\bar w=+\infty\]
for some $C^\infty$ functions~$\chi$ and~$\phi$ with $\lim_{|x|\to+\infty}\phi=+\infty$, $\chi\geq 0$ and $supp \chi$ compact.
\end{itemize}
\end{lemma}
\begin{proof}
For completeness, we report the arguments of~\cite{Lnt}.
As one can easily check, property~($ii$) obviously implies property~($i$) (possibly adding a constant).

Now, assuming~($i$), we want to prove~($ii$). We denote $K:=\ds\max_{\overline{B(0,R_0)}}|w|$ and $K_{\cal G}:=\ds\max_{\overline{B(0,R_0)}}|{\cal G}(w)|$.
We fix $\chi\in C^\infty_0(\re^N)$ such that $\chi\geq 0$, $\chi=1$ in $B(0,R_0)$ and $\hbox{supp} \chi\subset B(0,2R_0)$.
We claim that the function $w^\flat(x):=w(x)+K+K_{\cal G}+1$ satisfies
\begin{equation}\label{C2stella}
{\cal G}(w^\flat)+\chi w^\flat=:f^*(x)\geq 1 \quad \textrm{in }\re^N,\qquad  \lim_{|x|\to+\infty} w^\flat=+\infty.
\end{equation}
Indeed, the latter property is an immediate consequence of~($i$).\\
 Moreover, for $|x|\leq R_0$, we have
\[{\cal G}(w^\flat)+\chi w^\flat\geq -K_{\cal G}+\chi(w+K+K_{\cal G}+1)\geq 1\]
while, for $|x|\geq R_0$, we have
\[{\cal G}(w^\flat)+\chi w^\flat\geq {\cal G}(w)\geq 1;\]
hence, our claim~\eqref{C2stella} is proved.\\
Let us now consider a regular partition of unity $\ds\{\phi_i\}_{i\in \nat}$ such that  $\ds\phi_i\geq 0$, $\ds\sum_{i=1}^\infty \phi_i(x)=1,$
$ \ds\hbox{ supp }\phi_i\subset B(0,i+1)\setminus B(0,i-1)$, \\$\ds  \phi_i=1 \hbox { on }B(0,i+\frac{1}{2})\setminus B(0,i-\frac{1}{2})$.\\
We claim that there exists a regular solution to
\begin{equation}\label{cl:W_n}
{\cal G}(W_n)+\chi W_n=\sum_{i=1}^n\phi_i\qquad\textrm{in }\re^N, \qquad
0\leq W_n\leq w^\flat.
\end{equation}
In order to prove this existence, it is expedient to introduce, for $m\geq n+1$ and $\epsilon>0$, the following boundary value problems
\begin{equation}\label{W_nme}
\left\{\begin{array}{ll}
({\cal G}-\epsilon \Delta)(W_{nm}^{\epsilon})+\chi W_{nm}^{\epsilon}=\ds\sum_{i=1}^n\phi_i&\quad\textrm{in }B(0,m)\\
W_{nm}^{\epsilon}=0&\quad\textrm{on }\partial B(0,m).
\end{array}\right.
\end{equation}
By the non-degeneracy of the operator, the comparison principle applies to problems~\eqref{W_nme}. Hence, the Perron's method ensures that there exists a unique solution to~\eqref{W_nme}. By standard arguments in hypoelliptic theory (see \cite{Kri}, \cite{OR}), as $\epsilon \to 0^+$, $W_{nm}^{\epsilon}(x)$ converges to $W_{nm}(x)$ in $B(0,m)$, where $W_{nm}$ is the solution to
\begin{equation}\label{W_nm}
\left\{\begin{array}{ll}
{\cal G}(W_{nm})+\chi W_{nm}=\ds\sum_{i=1}^n\phi_i&\quad\textrm{in }B(0,m)\\
W_{nm}=0&\quad\textrm{on }\partial B(0,m),
\end{array}\right.
\end{equation}
where the boundary condition is attained only in the viscosity sense.
We observe that the H\"ormander's condition guarantees the comparison principle for~\eqref{W_nm}; since $0$ and $w^\flat$ are respectively a sub- and a supersolution,  there holds true $0\leq W_{nm}\leq w^\flat$ in $B(0,m)$. On the other hand, for $m_1>m$, still by comparison principle, we infer $W_{nm_1}^{\epsilon}(x)\geq W_{nm}^{\epsilon}(x)$ for every $x\in B(0,m)$; so, as $\epsilon \to 0^+$, we get $W_{nm_1}(x)\geq W_{nm}(x)$ for every $x\in B(0,m)$, namely, the sequence $\{W_{nm}\}_m$ is nondecreasing and locally bounded. Passing to the limit and using the regularity theory for hypoelliptic operators (see \cite{Bo}), we accomplish the proof of our claim~\eqref{cl:W_n}.

By~\eqref{cl:W_n}, the functions~$w_i(x):=W_i(x)-W_{i-1}(x)$ solve 
\[{\cal G}(w_i)+\chi w_i=\phi_i\qquad\textrm{in }\re^N\]
and verify: $\ds\sum_{i=1}^\infty w_i(x)<\infty$ in~$\re^N$.

Let us recall an elementary result: for any
$\sum_{i=1}^\infty a_i<+\infty$ with $a_i\geq 0$, there exists a sequence~$\{\lambda_i\}_i$ 
such that $\lim_{i\to+\infty}\lambda_i=+\infty$ and $\sum_{i=1}^\infty \lambda_i a_i<+\infty$.
Then in our case
there exists a sequence~$\{\lambda_i\}_i$ such that $\lim_{i\to+\infty}\lambda_i=+\infty$ and  $\ds\sum_{i=1}^\infty \lambda_i w_i(0)=K<+\infty$.\\
Let $n_0\in\N$ be fixed.  Let us denote by 
$\ds w^\sharp_n(x):=\ds\sum_{i=1}^n \lambda_i w_i(x)$. In $\ds B(0,n_0)$ $\ds w^\sharp_n(x)$
satisfies:
\begin{equation}\label{C2stellinan}
{\cal G}(w^\sharp_n)+\chi w^\sharp_n=\ds\sum_{i=1}^{n_0+1} \lambda_i\phi_i\qquad w^\sharp_n\geq 0
\end{equation}
and by Harnack inequality there exists a constant $C_{n_0}$ independent of $n$ such that
\begin{eqnarray*}
\sup_{B(0,\frac{n_0}2)} w^\sharp_n&\leq& C_{n_0} \big(\inf_{B(0,\frac{n_0}2)} w^\sharp_n+\sup_{B(0,\frac{n_0}2)} \sum_{i=1}^{n_0+1} \lambda_i\phi_i\big)\\
&\leq&  C_{n_0} \big(K+\sup_{B(0,\frac{n_0}2)} \sum_{i=1}^{n_0+1} \lambda_i\phi_i\big)=C_{n_0}^*
\end{eqnarray*}
This implies that in any bounded set $w^\sharp$ is well defined, i.e.
$w^\sharp(x):=\ds\sum_{i=1}^\infty \lambda_i w_i(x)<\infty$ for every $x\in \re^N$. Moreover the function~$w^\sharp$ satisfies
\begin{equation}\label{C2stellina}
{\cal G}(w^\sharp)+\chi w^\sharp=\ds\sum_{i=1}^\infty \lambda_i\phi_i=:\phi,\qquad w^\sharp\geq 0
\end{equation}
with $\lim_{x\to\infty}\phi(x)=+\infty$.\\
In conclusion, by~\eqref{C2stella} and~\eqref{C2stellina}, the function~$\bar w:=w^\flat+w^\sharp$ satisfies~($ii$).
\end{proof}

\noindent{\sc Acknowledgments } 
\\
We thank Italo Capuzzo Dolcetta and Olivier Ley for helpful discussions.\\
The first and the second authors are members of GNAMPA-INdAM.\\
The second author has been partially supported also by the "Progetto di Ateneo: Traffic flow on networks: analysis and control" of the University of Padova.\\
The third author has been partially funded by the ANR project ANR-12-BS01-
0008-01.


\noindent {\bf Address of the authors}\\
Paola Mannucci,\\
Dipartimento di Matematica,
Universit\`a degli Studi di Padova,\\
Via Trieste 63, 35131, Padova, Italy,\\
Claudio Marchi,\\
Dipartimento di Ingegneria dell'Informazione,
Universit\`a degli Studi di Padova,\\
Via Gradenigo 6, 35131, Padova, Italy,\\
Nicoletta Tchou, IRMAR, \\
Universit\`a de Rennes 1\\
Campus de Beaulieu, 35042 Rennes Cedex, France \\
mannucci@math.unipd.it, 
claudio.marchi@unipd.it,
nicoletta.tchou@univ-rennes1.fr	

\end{document}